\documentclass[12pt,a4wide]{amsart}
\usepackage{amsfonts}
\usepackage{latexsym,amsmath,amssymb}
\usepackage{graphics,epstopdf}
\usepackage[pdftex]{graphicx}
\usepackage{amsthm}
\usepackage{amssymb}
\usepackage{mathrsfs}
\usepackage{amsmath}
\usepackage{textcomp}
\usepackage{url}
\usepackage{array}
\usepackage{epsfig}
\usepackage{graphicx}
\usepackage[T1]{fontenc}
\usepackage[latin5]{inputenc}
\usepackage{float}
\usepackage{amssymb}
\usepackage{amsfonts}
\usepackage[centertags]{amsmath}
\usepackage{amsthm}
\usepackage{newlfont}

\theoremstyle{definition}

 \makeatletter
\theoremstyle{plain}
\theoremstyle{theorem}
\newtheorem{thm}{Theorem}[section]
\numberwithin{equation}{section}
\numberwithin{figure}{section}
\theoremstyle{plain}
\theoremstyle{proposition}

\theoremstyle{plain}
\theoremstyle{corollary}
\newtheorem{cor}[thm]{Corollary}
\theoremstyle{plain}
\theoremstyle{definition}

\theoremstyle{plain}
\theoremstyle{lemma}
\newtheorem{lem}[thm]{Lemma}
\theoremstyle{plain}
\theoremstyle{remark}

\begin{document}

\begin{center}
{\Large \textbf{Some approximation properties of bivariate Bleimann-Butzer
and Hahn operators based on $(p,q)$-integers}}

\bigskip

\textbf{M. Mursaleen} and \textbf{Md. Nasiruzzaman}

Department of\ Mathematics, Aligarh Muslim University, Aligarh--202002, India%
\\[0pt]

mursaleenm@gmail.com; nasir3489@gmail.com \\[0pt]

\bigskip

\bigskip

\textbf{Abstract}
\end{center}

\parindent=8mm {\footnotesize {Recently, Mursaleen et al applied $(p,q)$%
-calculus in approximation theory and introduced $(p,q)$-analogue of
Bernstein operators in \cite{mur7}. In this paper, we construct and
introduce a generalization of the bivariate Bleimann-Butzer-Hahn operators
based on $(p,q)$-integers and obtain Korovkin type approximation theorem of
these operators. Furthermore, we compute the rate of convergence of the
operators by using the modulus of continuity and Lipschitz type maximal
functions.}}

\bigskip



\section{Introduction and Preliminaries}

In 1980 Bleimann, Butzer and Hahn (BBH) \cite{brns} introduced the following
positive linear operators on the space of real functions defined on infinite
interval $[0,\infty )$:
\begin{equation*}
L_{n}(f;x)=\frac{1}{(1+x)^{n}}\sum_{k=0}^{n}f\left( \frac{k}{n-k+1}\right) %
\left[
\begin{array}{c}
n \\
k%
\end{array}%
\right] x^{k},x\geq 0~~~~~~~~~~~~~~~~~~~~~\eqno(1.1)
\end{equation*}

\parindent=8mm The approximation properties of the Bleimann-Butzer and Hahn
operators were studied by many authors, e.g. \cite{m1,m2,m3, brns}. In
approximation theory, $q$-type generalization of Bernstein polynomials was
introduced and studied by Lupaş \cite{lups} and Phillips \cite{philip}.

In \cite{ern},\cite{ali}, the BBH-type operators based on $q$-integers are
defined as follows:
\begin{equation*}
L_{n}^{q}(f;x)=\frac{1}{\ell _{n}(x)}\sum_{k=0}^{n}f\left( \frac{[k]_{q}}{%
[n-k+1]_{q}q^{k}}\right) q^{\frac{k(k-1)}{2}}\left[
\begin{array}{c}
n \\
k%
\end{array}%
\right] _{q}x^{k}~~~~~~~~~~~~~~\eqno(1.2)
\end{equation*}%
where $\ell _{n}(x)=\prod_{s=0}^{n-1}(1+q^{s}x)$.\newline

The bivariate case was introduced by D.D. Stancu \cite{dd} who studied the
bivariate Bernstein polynomials and estimated the order of approximation for
operators (1.2).\newline

Very recently, Mursaleen et al applied $(p,q)$-calculus in approximation
theory and introduced first $(p,q)$-analogue of Bernstein operators \cite%
{mur7} and studied approximation properties of $(p,q)$-analogue of
Bernstein-Stancu operators in \cite{mur8}. They also have introduced,
Bleimann-Butzer and Hahn operators \cite{et} based on $(p,q)$-integers as
follows:\newline

\begin{equation*}
L_{n}^{(p,q)}(f;x)=\frac{1}{\ell _{n}^{(p,q)}(x)}\sum_{k=0}^{n}f\left( \frac{%
p^{n-k+1}[k]_{p,q}}{[n-k+1]_{p,q}q^{k}}\right) p^{\frac{(n-k)(n-k-1)}{2}}q^{%
\frac{k(k-1)}{2}}\left[
\begin{array}{c}
n \\
k%
\end{array}%
\right] _{p,q}x^{k}~~~~~~~~~~~~~~~~~~~~~~~~\eqno(1.3)
\end{equation*}%
where, ~~$x\geq 0,~~0<q<p\leq 1$
\begin{equation*}
\frac{1}{\ell _{n}^{(p,q)}(x)}%
=\prod_{s=0}^{n-1}(p^{s}+q^{s}x),~~~~~~~~~~~~~~~~~~~~\eqno(1.4)
\end{equation*}%
and $f$ is defined on semiaxis $\mathbb{R}_{+}$.\newline
The Euler identity based on $(p,q)$-analogue is defined by
\begin{equation*}
\prod_{s=0}^{n-1}(p^{s}+q^{s}x)=\sum_{k=0}^{n}p^{\frac{(n-k)(n-k-1)}{2}}q^{%
\frac{k(k-1)}{2}}\left[
\begin{array}{c}
n \\
k%
\end{array}%
\right] _{p,q}x^{k}~~~~~~~~~~~~~~~~~~~~\eqno(1.5)
\end{equation*}

Let us recall certain notations on $(p,q)$-calculus.

The $(p,q)$ integers $[n]_{p,q}$ are defined by
\begin{equation*}
[n]_{p,q}=\frac{p^n-q^n}{p-q},~~~~~~~n=0,1,2,\cdots, ~~0<q<p\leq 1.
\end{equation*}
whereas $q$-integers are given by
\begin{equation*}
[n]_{q}=\frac{1-q^n}{1-q},~~~~~~~n=0,1,2,\cdots, ~~0<q< 1.
\end{equation*}%
\newline

It is very clear that $q$-integers and $(p,q)$-integers are different, that
is we cannot obtain $(p,q)$ integers just by replacing $q$ by $\frac{q}{p}$
in the definition of $q$-integers but if we put $p=1$ in definition of $%
(p,q) $ integers then $q$-integers become a particular case of $(p,q)$
integers. Thus we can say that $(p,q)$-calculus can be taken as a
generalization of $q$-calculus.\newline

Now by some simple calculation and induction on $n,$ we have $(p,q)$%
-binomial expansion as follows
\begin{equation*}
(ax+by)_{p,q}^{n}:=\sum\limits_{k=0}^{n}p^{\frac{(n-k)(n-k-1)}{2}}q^{\frac{%
k(k-1)}{2}} \left[
\begin{array}{c}
n \\
k%
\end{array}%
\right] _{p,q}a^{n-k}b^{k}x^{n-k}y^{k}
\end{equation*}
\begin{equation*}
(x+y)_{p,q}^{n}=(x+y)(px+qy)(p^2x+q^2y)\cdots (p^{n-1}x+q^{n-1}y)
\end{equation*}
\begin{equation*}
(1-x)_{p,q}^{n}=(1-x)(p-qx)(p^2-q^2x)\cdots (p^{n-1}-q^{n-1}x),
\end{equation*}%
\newline

and the $(p,q)$-binomial coefficients are defined by
\begin{equation*}
\left[
\begin{array}{c}
n \\
k%
\end{array}%
\right] _{p,q}=\frac{[n]_{p,q}!}{[k]_{p,q}![n-k]_{p,q}!}.
\end{equation*}

Again it can be easily verified that $(p,q)$-binomial expansion is different
from $q$-binomial expansion and is not a replacement of $q$ by $\frac{q}{p}$.%
\newline

By some simple calculation, we have the following relation
\begin{equation*}
q^{k}[n-k+1]_{p,q}=[n+1]_{p,q}-p^{n-k+1}[k]_{p,q}.~~~~~~~~~~~~~~~~~~~~\eqno%
(1.6)
\end{equation*}%
For details on $q$-calculus and $(p,q)$-calculus, one can refer \cite{ali,vp}%
, \cite{sad,vivek}, respectively.\newline

Now based on $(p,q)$-integers, we construct $(p,q)$-analogue of bivariate
BBH-operators, and we call it as $(p,q)$-bivariate Bleimann-Butzer-Hahn
operators and investigate its Korovokin type approximation properties, by
using the following test functions
\begin{equation*}
e_{ij}=\left( \frac{u}{1+u}\right) ^{i}\left( \frac{v}{1+v}\right)
^{j}~~~~~~~~~~~i,j=0,1,2.~~~~~~~~~~~~~~~~~~~~~\eqno(1.7)
\end{equation*}

Also for a space of generalized Lipschitz-type maximal functions, we give a
pointwise estimation.

\section{Construction of operators}

Let $\mathbb{R}^2_+ = [0,\infty)\times [0,\infty),~~f: \mathbb{R}^2_+ \to
\mathbb{R}$ and $0<q_{n_1},q_{n_2}<p_{n_1},p_{n_2} \leq 1$. We define the
bivariate extension of the $(p,q)$-Bleimann-Butzer and Hahn-type operators
based on $(p,q)$-integers as follows:\newline
\newline
$L^{(p_{n_1},p_{n_2}),(q_{n_1},q_{n_2})}_{n_1,n_2}(f;x,y)=\frac{1}{%
\ell^{(p_{n_1},q_{n_1})}_{n_1}(x)} \frac{1}{\ell^{(p_{n_2},q_{n_2})}_{n_2}(y)%
} \sum_{k_1=0}^{n_1}\sum_{k_2=0}^{n_2} f \left( \frac{%
p_{n_1}^{n_1-k_1+1}[k_1]_{p_{n_1},q_{n_1}}}{%
[n_1-k_1+1]_{p_{n_1},q_{n_1}}q_{n_1}^{k_1} }, \frac{%
p_{n_2}^{n_2-k_2+1}[k_2]_{p_{n_2},q_{n_2}}}{%
[n_2-k_2+1]_{p_{n_2},q_{n_2}}q_{n_2}^{k_2} }\right)$

\begin{equation*}
\times p_{n_{1}}^{\frac{(n_{1}-k_{1})(n_{1}-k_{1}-1)}{2}}q_{n_{1}}^{\frac{%
k_{1}(k_{1}-1)}{2}}p_{n_{2}}^{\frac{(n_{2}-k_{2})(n_{2}-k_{2}-1)}{2}%
}q_{n_{2}}^{\frac{k_{2}(k_{2}-1)}{2}}\left[
\begin{array}{c}
n_{1} \\
k_{1}%
\end{array}%
\right] _{p_{n_{1}},q_{n_{1}}}\left[
\begin{array}{c}
n_{2} \\
k_{2}%
\end{array}%
\right] _{p_{n_{2}},q_{n_{2}}}x^{k_{1}}y^{k_{2}}~~~~~~~~~~~~~~~~~~~~~~~~~~~%
\eqno(2.1)
\end{equation*}%
where
\begin{equation*}
\frac{1}{\ell _{n_{1}}^{(p_{n_{1}},q_{n_{1}})}(x)}%
=\prod_{s_{1}=0}^{n_{1}-1}(p_{n_{1}}^{s_{1}}+q_{n_{1}}^{s_{1}}x),~~~~~~~~%
\frac{1}{\ell _{n_{2}}^{(p_{n_{2}},q_{n_{2}})}(y)}%
=\prod_{s_{2}=0}^{n_{2}-1}(p_{n_{2}}^{s_{2}}+q_{n_{2}}^{s_{2}}y).
\end{equation*}%
It is easy to check that (2.1) is linear and positive. For $%
p_{n_{1}}=p_{n_{2}}=1$, the operators (2.1) are reduced to $q$-bivariate BBH
operators \cite{ern,biv}.

\begin{lem}
Let $e_{ij}:\mathbb{R}_{+}^{2}\rightarrow \lbrack 0,1)$ be the two
dimensional test functions defined in (1.7). Then the $(p,q)$-bivariate BBH
operators defined in (2.1) satisfy the following identities.

\begin{enumerate}
\item[$($i$)$] $%
L^{(p_{n_1},p_{n_2}),(q_{n_1},q_{n_2})}_{n_1,n_2}(e_{00};x,y)=1$

\item[$($ii$)$] $%
L^{(p_{n_1},p_{n_2}),(q_{n_1},q_{n_2})}_{n_1,n_2}(e_{10};x,y)=\frac{[n_1]_{P{%
n_1},q_{n_1}}}{[n_2+1]_{p_{n_1},q_{n_1}}} \frac{x}{1+x}$

\item[$($iii$)$] $%
L^{(p_{n_1},p_{n_2}),(q_{n_1},q_{n_2})}_{n_1,n_2}(e_{01};x,y)=\frac{[n_2]_{P{%
n_2},q_{n_2}}}{[n_2+1]_{p_{n_2},q_{n_2}}} \frac{y}{1+y}$

\item[$($iv$)$] $%
L^{(p_{n_1},p_{n_2}),(q_{n_1},q_{n_2})}_{n_1,n_2}(e_{20};x,y)= \frac{%
p_{n_1}q_{n_1}^2[n_1]_{p_{n_1},q_{n_1}}[n_1-1]_{p_{n_1},q_{n_1}}}{%
[n_1+1]_{p_{n_1},q_{n_1}}^2}\frac{x^2}{(1+x)(p_{n_1}+q_{n_1}x)} +\frac{%
p_{n_1}^{{n_1}+1}[n_1]_{p_{n_1},q_{n_1}}}{[n_1+1]_{p_{n_1},q_{n_1}}^2}\left(%
\frac{x}{1+x}\right)$

\item[$($v$)$] $%
L^{(p_{n_1},p_{n_2}),(q_{n_1},q_{n_2})}_{n_1,n_2}(e_{02};x,y)= \frac{%
p_{n_2}q_{n_2}^2[n_2]_{p_{n_2},q_{n_2}}[n_2-1]_{p_{n_2},q_{n_2}}}{%
[n_2+1]_{p_{n_2},q_{n_2}}^2}\frac{y^2}{(1+y)(p_{n_2}+q_{n_2}y)} +\frac{%
p_{n_2}^{{n_2}+1}[n_2]_{p_{n_2},q_{n_2}}}{[n_2+1]_{p_{n_2},q_{n_2}}^2}\left(%
\frac{y}{1+y}\right)$.
\end{enumerate}
\end{lem}

\begin{proof}

\begin{enumerate}
\item[$($i$)$] $%
L_{n_{1},n_{2}}^{(p_{n_{1}},p_{n_{2}}),(q_{n_{1}},q_{n_{2}})}\left(
e_{00};x,y\right) =\frac{1}{\ell _{n_{1}}^{(p_{n_{1}},q_{n_{1}})}(x)}\frac{1%
}{\ell _{n_{2}}^{(p_{n_{2}},q_{n_{2}})}(y)}\sum_{k_{1}=1}^{n_{1}}%
\sum_{k_{2}=1}^{n_{2}}$\newline
$\times p_{n_{1}}^{\frac{(n_{1}-k_{1})(n_{1}-k_{1}-1)}{2}}q_{n_{n_{1}}}^{%
\frac{k_{1}(k_{1}-1)}{2}}p_{n_{2}}^{\frac{(n_{2}-k_{2})(n_{2}-k_{2}-1)}{2}%
}q_{{n_{2}}}^{\frac{k_{2}(k_{2}-1)}{2}}\left[
\begin{array}{c}
n_{1} \\
k_{1}%
\end{array}%
\right] _{p_{n_{1}},q_{n_{1}}}\left[
\begin{array}{c}
n_{2} \\
k_{2}%
\end{array}%
\right] _{p_{n_{2}},q_{n_{2}}}x^{k_{1}}y^{k_{2}}$\newline

For $0<q<p\leq 1$, we have
\begin{equation*}
\sum_{k=0}^{n}p^{\frac{(n-k)(n-k-1)}{2}}q^{\frac{k(k-1)}{2}}\left[
\begin{array}{c}
n \\
k%
\end{array}%
\right] _{p,q}x^{k}=\prod_{k=0}^{n-1}(p^{k}+q^{k}x)=\frac{1}{\ell
_{n}^{p,q}(x)}.
\end{equation*}

\item[$($ii$)$] Let $u=\frac{%
p_{n_{1}}^{n_{1}-k_{1}+1}[k_{1}]_{p_{n_{1}},q_{n_{1}}}}{%
[n_{1}-k_{1}+1]_{p_{n_{1}},q_{n_{1}}}q_{n_{1}}^{k_{1}}}$. Then $\frac{u}{1+u}%
=\frac{[k_{1}]_{p_{n_{1}},q_{n_{1}}}p_{n_{1}}^{n_{1}-k_{1}+1}}{%
[n_{1}+1]_{p_{n_{1}},q_{n_{1}}}}$\newline
$L_{n_{1},n_{2}}^{(p_{n_{1}},p_{n_{2}}),(q_{n_{1}},q_{n_{2}})}\left(
e_{10};x,y\right) =\frac{1}{\ell _{n_{1}}^{(p_{n_{1}},q_{n_{1}})}(x)}\frac{1%
}{\ell _{n_{2}}^{(p_{n_{2}},q_{n_{2}})}(y)}\sum_{k_{1}=1}^{n_{1}}%
\sum_{k_{2}=1}^{n_{2}}\frac{%
[k_{1}]_{p_{n_{1}},q_{n_{1}}}p_{n_{1}}^{n_{1}-k_{1}+1}}{%
[n_{1}+1]_{p_{n_{1}},q_{n_{1}}}}\frac{%
[k_{2}]_{p_{n_{2}},q_{n_{2}}}p_{n_{2}}^{n_{2}-k_{2}+1}}{%
[n_{2}+1]_{p_{n_{2}},q_{n_{2}}}}$\newline
$\times p_{n_{1}}^{\frac{(n_{1}-k_{1})(n_{1}-k_{1}-1)}{2}}q_{n_{1}}^{\frac{%
k_{1}(k_{1}-1)}{2}}p_{n_{2}}^{\frac{(n_{2}-k_{2})(n_{2}-k_{2}-1)}{2}%
}q_{n_{2}}^{\frac{k_{2}(k_{2}-1)}{2}}\left[
\begin{array}{c}
n_{1} \\
k_{1}%
\end{array}%
\right] _{p_{n_{1}},q_{n_{1}}}\left[
\begin{array}{c}
n_{2} \\
k_{2}%
\end{array}%
\right] _{p_{n_{2}},q_{n_{2}}}x^{k_{1}}y^{k_{2}}$\newline
$=\frac{1}{\ell _{n_{1}}^{(p_{n_{1}},q_{n_{1}})}(x)}\frac{%
p_{n_{1}}[n_{1}]_{p_{n_{1}},q_{n_{1}}}}{[n_{1}+1]_{p_{n_{1}},q_{n_{1}}}}%
\sum_{k_{1}=1}^{n_{1}}p_{n_{1}}^{n_{1}-k_{1}+1+\frac{%
(n_{1}-k_{1})(n_{1}-k_{1}-1)}{2}}q_{n_{1}}^{\frac{k_{1}(k_{1}-1)}{2}}\left[
\begin{array}{c}
n_{1}-1 \\
k_{1}-1%
\end{array}%
\right] _{p_{n_{1}},q_{n_{1}}}(q_{n_{1}}x)^{k_{1}}$\newline
$=\frac{x}{\ell _{n_{1}}^{(p_{n_{1}},q_{n_{1}})}(x)}\frac{%
p_{n_{1}}[n_{1}]_{p_{n_{1}},q_{n_{1}}}}{[n_{1}+1]_{p_{n_{1}},q_{n_{1}}}}%
\sum_{k_{1}=0}^{n_{1}}p_{n_{1}-1}^{\frac{(n_{1}-k_{1})(n_{1}-k_{1}-1)}{2}%
}q_{n_{1}}^{\frac{k_{1}(k_{1}-1)}{2}}\left[
\begin{array}{c}
n_{1}-1 \\
k_{1}%
\end{array}%
\right] _{p_{n_{1}},q_{n_{1}}}(q_{n_{1}}x)^{k_{1}}$\newline
$=\frac{p_{n_{1}}[n_{1}]_{p_{n_{1}},q_{n_{1}}}}{%
[n_{1}+1]_{p_{n_{1}},q_{n_{1}}}}\frac{x}{1+x}.$\newline


\item[$($iii$)$] It can be proved in a similar way.

\item[$($iv$)$] $L_{p,q}^{n}\left( e_{20};x,y\right) =\frac{1}{\ell
_{n_{1}}^{(p_{n_{1}},q_{n_{1}})}(x)}\frac{1}{\ell
_{n_{2}}^{(p_{n_{2}},q_{n_{2}})}(y)}\sum_{k_{1}=1}^{n_{1}}%
\sum_{k_{2}=1}^{n_{2}}\frac{%
[k_{1}]_{p_{n_{1}},q_{n_{1}}}^{2}p_{n_{1}}^{2(n_{1}-k_{1}+1)}}{%
[n_{1}+1]_{p_{n_{1}},q_{n_{1}}}^{2}}p_{n_{1}}^{\frac{%
(n_{1}-k_{1})(n_{1}-k_{1}-1)}{2}}q_{n_{1}}^{\frac{k_{1}(k_{1}-1)}{2}}$%
\newline
$\times p_{n_{2}}^{\frac{(n_{2}-k_{2})(n_{2}-k_{2}-1)}{2}}q_{n_{2}}^{\frac{%
k_{2}(k_{2}-1)}{2}}\left[
\begin{array}{c}
n_{1} \\
k_{1}%
\end{array}%
\right] _{p_{n_{1}},q_{n_{1}}}\left[
\begin{array}{c}
n_{2} \\
k_{2}%
\end{array}%
\right] _{p_{n_{1}},q_{n_{1}}}x^{k_{1}}y^{k_{2}}$.\newline
Using (1.6), we get\newline

$L_{p,q}^{n}\left( \frac{u^{2}}{(1+u)^{2}};x\right) $\newline
$=\frac{1}{\ell _{n_{1}}^{(p_{n_{1}},q_{n_{1}})}(x)}\sum_{k_{1}=2}^{n_{1}}%
\frac{%
q_{n_{1}}[k_{1}]_{p_{n_{1}},q_{n_{1}}}[k_{1}-1]_{p_{n_{1}},q_{n_{1}}}p_{n_{1}}^{2n_{1}-2k_{1}+2}%
}{[n_{1}+1]_{p_{n_{1}},q_{n_{1}}}^{2}}p_{n_{1}}^{\frac{%
(n_{1}-k_{1})(n_{1}-k_{1}-1)}{2}}q_{n_{1}}^{\frac{k_{1}(k_{1}-1)}{2}}\left[
\begin{array}{c}
n_{1} \\
k_{1}%
\end{array}%
\right] _{p_{n_{1}},q_{n_{1}}}x^{k_{1}}$\newline
\begin{equation*}
+\frac{1}{\ell _{n_{1}}^{(p_{n_{1}},q_{n_{1}})}(x)}%
\sum_{k_{1}=1}^{n_{1}}p_{n_{1}}^{k_{1}-1}\frac{%
[k_{1}]_{p_{n_{1}},q_{n_{1}}}p_{n_{1}}^{2n_{1}-2k_{1}+2}}{%
[n_{1}+1]_{p_{n_{1}},q_{n_{1}}}^{2}}p_{n_{1}}^{\frac{%
(n_{1}-k_{1})(n_{1}-k_{1}-1)}{2}}q_{n_{1}}^{\frac{k_{1}(k_{1}-1)}{2}}\left[
\begin{array}{c}
n_{1} \\
k_{1}%
\end{array}%
\right] _{p_{n_{1}},q_{n_{1}}}x^{k_{1}}
\end{equation*}

\begin{equation*}
=\frac{1}{\ell _{n_{1}}^{(p_{n_{1}},q_{n_{1}})}(x)}\frac{%
q_{n_{1}}[n_{1}]_{p_{n_{1}},q_{n_{1}}}[n_{1}-1]_{p_{n_{1}},q_{n_{1}}}}{%
[n_{1}+1]_{p_{n_{1}},q_{n_{1}}}^{2}}\sum_{k_{1}=2}^{n_{1}}p_{n_{1}}^{\left(
(2n_{1}-2k_{1}+2)+\frac{(n_{1}-k_{1})(n_{1}-k_{1}-1)}{2}\right) }q_{n_{1}}^{%
\frac{k_{1}(k_{1}-1)}{2}}\left[
\begin{array}{c}
n_{1}-2 \\
k_{1}-2%
\end{array}%
\right] _{p_{n_{1}},q_{n_{1}}}x^{k_{1}}
\end{equation*}

\begin{equation*}
+\frac{1}{\ell _{n_{1}}^{(p_{n_{1}},q_{n_{1}})}(x)}\frac{%
[n_{1}]_{p_{n_{1}},q_{n_{1}}}}{[n_{1}+1]_{p_{n_{1}},q_{n_{1}}}^{2}}%
\sum_{k_{1}=1}^{n_{1}}p_{n_{1}}^{\left( (k_{1}-1)+(2n_{1}-2k_{1}+2)+\frac{%
(n_{1}-k_{1})(n_{1}-k_{1}-1)}{2}\right) }q_{n_{n_{1}}}^{\frac{k_{1}(k_{1}-1)%
}{2}}\left[
\begin{array}{c}
n_{1}-1 \\
k_{1}-1%
\end{array}%
\right] _{p_{n_{1}},q_{n_{1}}}x^{k_{1}}
\end{equation*}

\begin{equation*}
=x^{2}\frac{1}{\ell _{n_{1}}^{(p_{n_{1}},q_{n_{1}})}(x)}\frac{%
q_{n_{1}}[n_{1}]_{p_{n_{1}},q_{n_{1}}}[n_{1}-1]_{p_{n_{1}},q_{n_{1}}}}{%
[n_{1}+1]_{p_{n_{1}},q_{n_{1}}}^{2}}\sum_{k_{1}=0}^{n_{1}-2}p_{n_{1}}^{%
\left( (2n_{1}-2k_{1}-2)+\frac{(n_{1}-k_{1}-2)(n_{1}-k_{1}-3)}{2}\right)
}q_{n_{1}}^{\frac{(k_{1}+1)(k_{1}+2)}{2}}\left[
\begin{array}{c}
n_{1}-2 \\
k_{1}%
\end{array}%
\right] _{p_{n_{1}},q_{n_{1}}}x^{k_{1}}
\end{equation*}

\begin{equation*}
+x\frac{1}{\ell _{n_{1}}^{(p_{n_{1}},q_{n_{1}})}(x)}\frac{%
[n_{1}]_{p_{n_{1}},q_{n_{1}}}}{[n_{1}+1]_{p_{n_{1}},q_{n_{1}}}^{2}}%
\sum_{k_{1}=0}^{n_{1}-1}p_{n_{1}}^{\left( k_{1}+(2n_{1}-2k_{1})+\frac{%
(n_{1}-k_{1}-1)(n_{1}-k_{1}-2)}{2}\right) }q_{n_{1}}^{\frac{k_{1}(k_{1}+1)}{2%
}}\left[
\begin{array}{c}
n_{1}-1 \\
k_{1}%
\end{array}%
\right] _{p_{n_{1}},q_{n_{1}}}x^{k_{1}}
\end{equation*}

\begin{equation*}
=x^{2}\frac{1}{\ell _{n_{1}}^{(p_{n_{1}},q_{n_{1}})}(x)}\frac{%
p_{n_{1}}q_{n_{1}}^{2}[n_{1}]_{p_{n_{1}},q_{n_{1}}}[n_{1}-1]_{p_{n_{1}},q_{n_{1}}}%
}{[n_{1}+1]_{p_{n_{1}},q_{n_{1}}}^{2}}\sum_{k_{1}=0}^{n_{1}-2}p_{n_{1}}^{%
\frac{(n_{1}-k_{1})(n_{1}-k_{1}-1)}{2}}q_{{n_{1}}}^{\frac{k_{1}(k_{1}-1)}{2}}%
\left[
\begin{array}{c}
n_{1}-2 \\
k_{1}%
\end{array}%
\right] _{p_{n_{1}},q_{n_{1}}}(q_{n_{1}}^{2}x)^{k_{1}}
\end{equation*}

\begin{equation*}
+x\frac{1}{\ell _{n_{1}}^{(p_{n_{1}},q_{n_{1}})}(x)}\frac{%
p_{n_{1}}^{n_{1}+1}[n_{1}]_{p_{n_{1}},q_{n_{1}}}}{%
[n_{1}+1]_{p_{n_{1}},q_{n_{1}}}^{2}}\sum_{k_{1}=0}^{n_{1}-1}p_{n_{1}}^{\frac{%
(n_{1}-k_{1})(n_{1}-k_{1}-1)}{2}}q_{n_{1}}^{\frac{k_{1}(k_{1}-1)}{2}}\left[
\begin{array}{c}
n_{1}-1 \\
k_{1}%
\end{array}%
\right] _{p_{n_{1}},q_{n_{1}}}(q_{n_{1}}x)^{k_{1}}
\end{equation*}

\begin{equation*}
=\frac{%
p_{n_{1}}q_{n_{1}}^{2}[n_{1}]_{p_{n_{1}},q_{n_{1}}}[n_{1}-1]_{p_{n_{1}},q_{n_{1}}}%
}{[n_{1}+1]_{p_{n_{1}},q_{n_{1}}}^{2}}\frac{x^{2}}{%
(1+x)(p_{n_{1}}+q_{n_{1}}x)}+\frac{%
p_{n_{1}}^{n_{1}+1}[n_{1}]_{p_{n_{1}},q_{n_{1}}}}{%
[n_{1}+1]_{p_{n_{1}},q_{n_{1}}}^{2}}\left( \frac{x}{1+x}\right) .
\end{equation*}

\item[$($v$)$] It can be proved in a similar way.
\end{enumerate}
\end{proof}

\begin{lem}
The operators defined in (2.1) satisfy the following conditions.

\begin{enumerate}
\item[$($i$)$] $%
L^{(p_{n_1},p_{n_2}),(q_{n_1},q_{n_2})}_{n_1,n_2}(f;x,y)=A^{(p_{n_1},q_{n_1}),x}_{n_1}\left( B^{(p_{n_2},q_{n_2}),y}_{n_2}(f;x,y)\right)
$

\item[$($ii$)$] $%
L^{(p_{n_1},p_{n_2}),(q_{n_1},q_{n_2})}_{n_1,n_2}(f;x,y)=B^{(p_{n_2},q_{n_2}),y}_{n_2}\left( A^{(p_{n_1},q_{n_1}),x}_{n_1}(f;x,y)\right)
$
\end{enumerate}
\end{lem}

\begin{proof}

\begin{enumerate}
\item[$($i$)$] $A_{n_{1}}^{(p_{n_{1}},q_{n_{1}}),x}\left(
B_{n_{2}}^{(p_{n_{2}},q_{n_{2}}),y}(f;x,y)\right) $\newline
$=A_{n_{1}}^{(p_{n_{1}},q_{n_{1}}),x}\left( \frac{1}{\ell
_{n_{2}}^{(p_{n_{2}},q_{n_{2}})}(y)}\sum_{k_{2}=0}^{n_{2}}f\left( x,\frac{%
p_{n_{2}}^{n_{2}-k_{2}+1}[k_{2}]_{p_{n_{2}},q_{n_{2}}}}{%
[n_{2}-k_{2}+1]_{p_{n_{2}},q_{n_{2}}}q_{n_{2}}^{k_{2}}}\right) p_{n_{2}}^{%
\frac{(n_{2}-k_{2})(n_{2}-k_{2}-1)}{2}}q_{n_{2}}^{\frac{k_{2}(k_{2}-1)}{2}}%
\left[
\begin{array}{c}
n_{2} \\
k_{2}%
\end{array}%
\right] _{p_{n_{2}},q_{n_{2}}}y^{k_{2}}\right) $\newline
$=\frac{1}{\ell _{n_{2}}^{(p_{n_{2}},q_{n_{2}})}(y)}%
\sum_{k_{2}=0}^{n_{2}}A_{n_{1}}^{(p_{n_{1}},q_{n_{1}}),x}\left( f\left( x,%
\frac{p_{n_{2}}^{n_{2}-k_{2}+1}[k_{2}]_{p_{n_{2}},q_{n_{2}}}}{%
[n_{2}-k_{2}+1]_{p_{n_{2}},q_{n_{2}}}q_{n_{2}}^{k_{2}}}\right) \right)
p_{n_{2}}^{\frac{(n_{2}-k_{2})(n_{2}-k_{2}-1)}{2}}q_{n_{2}}^{\frac{%
k_{2}(k_{2}-1)}{2}}\left[
\begin{array}{c}
n_{2} \\
k_{2}%
\end{array}%
\right] _{p_{n_{2}},q_{n_{2}}}y^{k_{2}}$\newline

$=\frac{1}{\ell _{n_{2}}^{(p_{n_{2}},q_{n_{2}})}(y)}%
\sum_{k_{2}=0}^{n_{2}}p_{n_{2}}^{\frac{(n_{2}-k_{2})(n_{2}-k_{2}-1)}{2}%
}q_{n_{2}}^{\frac{k_{2}(k_{2}-1)}{2}}\left[
\begin{array}{c}
n_{2} \\
k_{2}%
\end{array}%
\right] _{p_{n_{2}},q_{n_{2}}}y^{k_{2}}$\newline

$\times \sum_{k_{1}=0}^{n_{1}}\frac{1}{\ell
_{n_{1}}^{(p_{n_{1}},q_{n_{1}})}(x)}f\left( \frac{%
p_{n_{1}}^{n_{1}-k_{1}+1}[k_{1}]_{p_{n_{1}},q_{n_{1}}}}{%
[n_{1}-k_{1}+1]_{p_{n_{1}},q_{n_{1}}}q_{n_{1}}^{k_{1}}},\frac{%
p_{n_{2}}^{n_{2}-k_{2}+1}[k_{2}]_{p_{n_{2}},q_{n_{2}}}}{%
[n_{2}-k_{2}+1]_{p_{n_{2}},q_{n_{2}}}q_{n_{2}}^{k_{2}}}\right) $\newline

$\times p_{n_1}^{\frac{(n_1-k_1)(n_1-k_1-1)}{2}}q_{n_1}^{\frac{k_1(k_1-1)}{2}%
} \left[
\begin{array}{c}
n_1 \\
k_1%
\end{array}%
\right] _{p_{n_1},q_{n_1}} x^{k_1} $\newline

$=\frac{1}{\ell^{(p_{n_1},q_{n_1})}_{n_1}(x)} \frac{1}{%
\ell^{(p_{n_2},q_{n_2})}_{n_2}(y)} \sum_{k_1=0}^{n_1}\sum_{k_2=0}^{n_2} f
\left( \frac{p_{n_1}^{n_1-k_1+1}[k_1]_{p_{n_1},q_{n_1}}}{%
[n_1-k_1+1]_{p_{n_1},q_{n_1}}q_{n_1}^{k_1} }, \frac{%
p_{n_2}^{n_2-k_2+1}[k_2]_{p_{n_2},q_{n_2}}}{%
[n_2-k_2+1]_{p_{n_2},q_{n_2}}q_{n_2}^{k_2} }\right)$\newline
$\times p_{n_1}^{\frac{(n_1-k_1)(n_1-k_1-1)}{2}}q_{n_1}^{\frac{k_1(k_1-1)}{2}%
} p_{n_2}^{\frac{(n_2-k_2)(n_2-k_2-1)}{2}}q_{n_2}^{\frac{k_2(k_2-1)}{2}} %
\left[
\begin{array}{c}
n_1 \\
k_1%
\end{array}%
\right] _{p_{n_1},q_{n_1}} \left[
\begin{array}{c}
n_2 \\
k_2%
\end{array}%
\right] _{p_{n_2},q_{n_2}} x^{k_1}y^{k_2} $\newline

$=L^{(p_{n_1},p_{n_2}),(q_{n_1},q_{n_2})}_{n_1,n_2}(f;x,y)$.\newline

\item[$($ii$)$] It can be proved in a similar way.
\end{enumerate}
\end{proof}

\section{Approximation properties of Bivariate operators}

In this section, we obtain the Korovkin's type approximation theorem for our
operators defined in (2.1).

Let $C_{\mathcal{B}}(\mathbb{R}_{+}^{2})$ be the set of all bounded and
continuous functions on $\mathbb{R}_{+}^{2}$ which is linear normed space
with
\begin{equation*}
\parallel f\parallel _{C_{B}(\mathbb{R}_{+}^{2})}=\sup_{x,y\geq 0}\mid
f(x,y)\mid .
\end{equation*}%
If
\begin{equation*}
\lim_{n,m\rightarrow \infty }\parallel f_{n,m}-f\parallel _{C_{\mathcal{B}}(%
\mathbb{R}_{+}^{2})}=0
\end{equation*}%
holds, then we say that the sequence $\{f_{n,m}\}$ converges uniformly to $f$
and it is written as $f_{n,m}\rightrightarrows f$.

Now, let us introduce modulus of continuity $\omega (\delta )$ which satisfy
the following conditions:

\begin{enumerate}
\item[$($i$)$] $\omega(\delta)$ is a non-negative increasing function for $%
\delta$ on $\mathbb{R}_+$

\item[$($ii$)$] $\omega(\delta_1+\delta_2)\leq
\omega(\delta_1)+\omega(\delta_2)$

\item[$($iii$)$] $\lim_{\delta \to 0}\omega(\delta)=0$.
\end{enumerate}

Let ${H}_{\omega }$ be the space of all real-valued functions $f$ satisfying
the condition
\begin{equation*}
\mid f(x)-f(y)\mid \leq \omega \left( \bigg{|}\frac{x}{1+x}-\frac{y}{1+y}%
\bigg{|}\right) ,
\end{equation*}%
for any $x,y\in \mathbb{R}_{+}.$\newline
Note that ${H}_{\omega }\subset C_{\mathcal{B}}(\mathbb{R}_{+})$ for the
bounded and continuous function $f$ on $\mathbb{R}_{+}.$

\begin{thm}[\protect\cite{butz}]
Let $\{A_n\}$ be the sequence of positive linear operators from $H_\omega(%
\mathbb{R}_+)$ into $C_B(\mathbb{R}_+)$, satisfying the conditions
\begin{equation*}
\lim_{n \to \infty} \parallel A_n \left( \left( \frac{t}{1+t}%
\right)^\nu;x\right)-\left(\frac{x}{1+x}\right)^\nu \parallel_{C_{B}},
\end{equation*}
for $\nu=0,1,2$. Then for any function $f \in H_\omega$
\begin{equation*}
\lim_{n \to \infty} \parallel A_n (f)-f \parallel_{C_{B}}=0.
\end{equation*}
\end{thm}

In order to get the convergence results for the operators $%
L_{n_{1},n_{2}}^{(p_{n_{1}},p_{n_{2}}),(q_{n_{1}},q_{n_{2}})}$, we take $%
p=p_{n_{1}},~~p=p_{n_{2}}$ and $q=q_{n_{1}},~~q=q_{n_{2}}$ where $%
q_{n_{i}}\in (0,1)$ and $p_{n_{i}}\in (q_{n_{i}},1]$ for $i=1,2$ satisfying,
\begin{equation*}
\lim_{n_{1},n_{2}\rightarrow \infty }p_{n_{1}},~~p_{n_{2}}\rightarrow 1,~~~~%
\mbox{and}~~~\lim_{n_{1},n_{2}\rightarrow \infty
}q_{n_{1}},~~q_{n_{2}}\rightarrow 1~~~~~~~~~~~~~~~~~~~~~\eqno(3.1)
\end{equation*}

Now we prove the bivariate case of Theorem 3.1.

\begin{thm}
Let $p=p_{n_1},~~p=p_{n_2}$ and $q=q_{n_1},~~~q=q_{n_2}$ satisfying (3.1),
for $0<q_{n_1}<p_{n_1}\leq 1,~~0<q_{n_2}<p_{n_2}\leq 1$ and if $%
A^{(p_{n_1},p_{n_2}),(q_{n_1},q_{n_2})}_{n_1,n_2}$ be the sequence of a
linear positive operator, defined as $%
A^{(p_{n_1},p_{n_2}),(q_{n_1},q_{n_2})}_{n_1,n_2} :H_\omega(\mathbb{R}_+^2)
\to C_B(\mathbb{R}_+^2)$ , satisfying the following conditions:

\begin{enumerate}
\item[$($i$)$] $\lim_{n_1,n_2 \to \infty}\parallel
A^{(p_{n_1},p_{n_2}),(q_{n_1},q_{n_2})}_{n_1,n_2}
(e_{00};x,y)-e_{00}\parallel_{C_{\mathbb{R}_+^2}}=0$,

\item[$($ii$)$] $\lim_{n_1,n_2 \to \infty}\parallel
A^{(p_{n_1},p_{n_2}),(q_{n_1},q_{n_2})}_{n_1,n_2}(e_{10};x,y)-e_{10}%
\parallel_{C_{\mathbb{R}_+^2}}=0$,

\item[$($iii$)$] $\lim_{n_1,n_2 \to \infty}\parallel
A^{(p_{n_1},p_{n_2}),(q_{n_1},q_{n_2})}_{n_1,n_2}(e_{01};x,y)-e_{01}%
\parallel_{C_{\mathbb{R}_+^2}}=0$,

\item[$($iv$)$] $\lim_{n_1,n_2 \to \infty}\parallel
A^{(p_{n_1},p_{n_2}),(q_{n_1},q_{n_2})}_{n_1,n_2}(e_{20};x,y)-e_{20}%
\parallel_{C_{\mathbb{R}_+^2}}=0$,

\item[$($v$)$] $\lim_{n_1,n_2 \to \infty}\parallel
A^{(p_{n_1},p_{n_2}),(q_{n_1},q_{n_2})}_{n_1,n_2}(e_{02};x,y)-e_{02}%
\parallel_{C_{\mathbb{R}_+^2}}=0.$
\end{enumerate}
Then for any function $f\in H_{\omega }(\mathbb{R}_{+}^{2})$
\begin{equation*}
\lim_{n_{1},n_{2}\rightarrow \infty }\parallel
A_{n_{1},n_{2}}^{(p_{n_{1}},p_{n_{2}}),(q_{n_{1}},q_{n_{2}})}(f;x,y)-f(x,y)%
\parallel _{C_{\mathbb{R}_{+}^{2}}}=0\eqno(3.2)
\end{equation*}%
holds. Here $e_{ij}:\mathbb{R}_{+}^{2}\rightarrow \lbrack 0,1)$ defined in
(1.7) are two dimensional test functions.
\end{thm}

\begin{proof}
For the bivariate case on ${H}_{\omega }(\mathbb{R}_{+}^{2})$, we have
\begin{equation*}
\mid f(u,v)-f(x,y)\mid \leq \omega \left( \bigg{|}\left( \frac{u}{1+u},\frac{%
v}{1+v}\right) -\left( \frac{x}{1+x},\frac{y}{1+y}\right) \bigg{|}\right) .
\end{equation*}%
If $f\in {H}_{\omega }(\mathbb{R}_{+}^{2})$, then for any $\epsilon >0$,
there exits $\delta >0$ such that
\begin{equation*}
\mid f(u,v)-f(x,y)\mid <\epsilon .
\end{equation*}%
Also if
\begin{equation*}
\sqrt{\left( \frac{u}{1+u}-\frac{x}{1+x}\right) ^{2}+\left( \frac{v}{1+v}-%
\frac{y}{1+y}\right) ^{2}}<\delta
\end{equation*}%
boundedness of $f$ implies that there exits a positive constant ${M}$ such
that
\begin{equation*}
\mid f(u,v)-f(x,y)\mid \leq \frac{2{M}}{\delta ^{2}}\left[ \left( \frac{u}{%
1+u}-\frac{x}{1+x}\right) ^{2}+\left( \frac{v}{1+v}-\frac{y}{1+y}\right) ^{2}%
\right] ,
\end{equation*}%
if $\sqrt{\left( \frac{u}{1+u}-\frac{x}{1+x}\right) ^{2}+\left( \frac{v}{1+v}%
-\frac{y}{1+y}\right) ^{2}}\geq \delta .$

Therefore, for all $(u,v),~~(x,y)\in \mathbb{R}_{+}^{2}$, we have
\begin{equation*}
\mid f(u,v)-f(x,y)\mid \leq \epsilon +\frac{2{M}}{\delta ^{2}}\left[ \left(
\frac{u}{1+u}-\frac{x}{1+x}\right) ^{2}+\left( \frac{v}{1+v}-\frac{y}{1+y}%
\right) ^{2}\right] .
\end{equation*}%
Now applying the operator $%
A_{n_{1},n_{2}}^{(p_{n_{1}},p_{n_{2}}),(q_{n_{1}},q_{n_{2}})}$ to the above
inequality, we get\newline
\newline
$\mid
A_{n_{1},n_{2}}^{(p_{n_{1}},p_{n_{2}}),(q_{n_{1}},q_{n_{2}})}(f;x,y)-f\mid
_{C_{\mathbb{R}_{+}^{2}}}$\newline
$\leq (\epsilon +{M})\mid
A_{n_{1},n_{2}}^{(p_{n_{1}},p_{n_{2}}),(q_{n_{1}},q_{n_{2}})}(e_{00};x,y)-e_{00}\mid +\epsilon
$\newline
$+\frac{2{M}}{\delta ^{2}}\mid
A_{n_{1},n_{2}}^{(p_{n_{1}},p_{n_{2}}),(q_{n_{1}},q_{n_{2}})}(e_{20}+e_{02};x,y)-(e_{20}+e_{02})\mid
$\newline
$+\frac{4{M}}{\delta ^{2}}\mid
A_{n_{1},n_{2}}^{(p_{n_{1}},p_{n_{2}}),(q_{n_{1}},q_{n_{2}})}(e_{10};x,y)-e_{10}\mid +%
\frac{4{M}}{\delta ^{2}}\mid
A_{n_{1},n_{2}}^{(p_{n_{1}},p_{n_{2}}),(q_{n_{1}},q_{n_{2}})}(e_{01};x,y)-e_{01}\mid .
$\newline
By using the conditions (i)-(v) we get the result.
\end{proof}

\begin{thm}
Let $p=p_{n_{1}},~~p=p_{n_{2}}$ and $q=q_{n_{1}},~~~q=q_{n_{2}}$ satisfying
(3.1), for $0<q_{n_{1}}<p_{n_{1}}\leq 1,~~0<q_{n_{2}}<p_{n_{2}}\leq 1$ and
if $L_{n_{1},n_{2}}^{(p_{n_{1}},p_{n_{2}}),(q_{n_{1}},q_{n_{2}})}$ be the
sequence of a linear positive operator defined in (2.1) such that $%
L_{n_{1},n_{2}}^{(p_{n_{1}},p_{n_{2}}),(q_{n_{1}},q_{n_{2}})}:H_{\omega }(%
\mathbb{R}_{+}^{2})\rightarrow C_{B}(\mathbb{R}_{+}^{2})$ , satisfying the
conditions of Theorem 3.2. Then for any $f\in H_{\omega }(\mathbb{R}_{+}^{2})
$, $L_{n_{1},n_{2}}^{(p_{n_{1}},p_{n_{2}}),(q_{n_{1}},q_{n_{2}})}$ converges
uniformly to $f$. That is $\forall f\in H_{\omega }(\mathbb{R}_{+}^{2})$, we
have
\begin{equation*}
\lim_{n_{1},n_{2}\rightarrow \infty }\parallel
L_{n_{1},n_{2}}^{(p_{n_{1}},p_{n_{2}}),(q_{n_{1}},q_{n_{2}})}(f;x,y)-f(x,y)%
\parallel _{C_{\mathbb{R}_{+}^{2}}}=0.
\end{equation*}
\end{thm}

\begin{proof}
Using Theorem 3.1 we see that it is sufficient to verify the following
conditions:
\begin{equation*}
\lim_{n_{1},n_{2}\rightarrow \infty }\parallel
L_{n_{1},n_{2}}^{(p_{n_{1}},p_{n_{2}}),(q_{n_{1}},q_{n_{2}})}(e_{ij};x,y)-e_{ij}\parallel _{C_{%
\mathbb{R}_{+}^{2}}}=0,~~i,j=0,1,2,~~~~~~~~~~~~~~~~~~~~~~~~~~\eqno(3.3)
\end{equation*}%
here $e_{ij}$ are the test functions defined in (1.7). From Lemma 2.1, the
first condition of (3.3) is fulfilled for $i,j=0$. Now it is easy to see
that from (ii) of Lemma 2.1 \newline
\begin{eqnarray*}
\lim_{n_{1},n_{2}\rightarrow \infty }\parallel
L_{n_{1},n_{2}}^{(p_{n_{1}},p_{n_{2}}),(q_{n_{1}},q_{n_{2}})}(e_{10};x,y)-e_{10}\parallel _{C_{%
\mathbb{R}_{+}^{2}}} &=&\sup_{x,y\geq 0}\bigg{|}\frac{%
p_{n_{1}}[n_{1}]_{p_{n_{1}},q_{n_{1}}}}{[n_{1}+1]_{p_{n_{1}},q_{n_{1}}}}%
\frac{x}{1+x}-\frac{x}{1+x}\bigg{|} \\
&\leq &\bigg{|}\frac{p_{n_{1}}[n_{1}]_{p_{n_{1}},q_{n_{1}}}}{%
[n_{1}+1]_{p_{n_{1}},q_{n_{1}}}}-1\bigg{|} \\
&\leq &\bigg{|}\left( \frac{p_{n_{1}}}{q_{n_{1}}}\right) \left(
1-p_{n_{1}}^{n_{1}}\frac{1}{[n_{1}+1]_{p_{n_{1}},q_{n_{1}}}}\right) -1%
\bigg{|}
\end{eqnarray*}%
since $%
q_{n_{1}}[n_{1}]_{p_{n_{1}},q_{n_{1}}}=[n_{1}+1]_{p_{n_{1}},q_{n_{1}}}-p_{n_{1}}^{n_{1}},~~[n_{1}+1]_{p_{n_{1}},q_{n_{1}}}\rightarrow \infty
$ as $n_{1}\rightarrow \infty $, the condition (3.3) holds for $i=1,~~j=0$.%
\newline
Similarly for $i=0,~~j=1$ it can be shown easily.

To verify this condition for $i=2,~~j=0$, consider (iv) of Lemma 2.1. Then
we see that\newline
$\lim_{n_{1}\rightarrow \infty }\parallel
L_{n_{1},n_{2}}^{(p_{n_{1}},p_{n_{2}}),(q_{n_{1}},q_{n_{2}})}(e_{20};x,y)-e_{20}\parallel _{C_{%
\mathbb{R}_{+}^{2}}}$\newline
$=\sup_{x\geq 0}\left\{ \frac{x^{2}}{(1+x)^{2}}\left( \frac{%
p_{n_{1}}q_{n_{1}}^{2}[n_{1}]_{p_{n_{1}},q_{n_{1}}}[n_{1}-1]_{p_{n_{1}},q_{n_{1}}}%
}{[n_{1}+1]_{p_{n_{1}},q_{n_{1}}}^{2}}.\frac{1+x}{p_{n_{1}}+q_{n_{1}}x}%
-1\right) +\frac{p_{n_{1}}^{n_{1}+1}[n_{1}]_{p_{n_{1}},q_{n_{1}}}}{%
[n_{1}+1]_{p_{n_{1}},q_{n_{1}}}^{2}}.\frac{x}{1+x}\right\} $.\newline
A small calculation leads to
\begin{equation*}
\frac{\lbrack n_{1}]_{p_{n_{1}},q_{n_{1}}}[n_{1}-1]_{p_{n_{1}},q_{n_{1}}}}{%
[n_{1}+1]_{p_{n_{1}},q_{n_{1}}}^{2}}=\frac{1}{q_{n_{1}}^{3}}\left\{
1-p_{n_{1}}^{n_{1}}\left( 2+\frac{q_{n_{1}}}{p_{n_{1}}}\right) \frac{1}{%
[n_{1}+1]_{p_{n_{1}},q_{n_{1}}}}+(p_{n_{1}}^{n_{1}})^{2}\left( 1+\frac{%
q_{n_{1}}}{p_{n_{1}}}\right) \frac{1}{[n_{1}+1]_{p_{n_{1}},q_{n_{1}}}^{2}}%
\right\} ,
\end{equation*}%
and
\begin{equation*}
\frac{\lbrack n_{1}]_{p_{n_{1}},q_{n_{1}}}}{%
[n_{1}+1]_{p_{n_{1}},q_{n_{1}}}^{2}}=\frac{1}{q_{n_{1}}}\left( \frac{1}{%
[n_{1}+1]_{p_{n_{1}},q_{n_{1}}}}-p_{n_{1}}^{n_{1}}\frac{1}{%
[n_{1}+1]_{p_{n_{1}},q_{n_{1}}}^{2}}\right) .
\end{equation*}

Thus we have\newline
\newline
$\lim_{n_1 \to \infty}\parallel
L^{(p_{n_1},p_{n_2}),(q_{n_1},q_{n_2})}_{n_1,n_2}
(e_{20};x,y)-e_{20}\parallel_{C_{\mathbb{R}_+^2}}$\newline

$\leq \frac{p_{n_{1}}}{q_{n_{1}}}\left\{ 1-p_{n_{1}}^{n_{1}}\left( 2+\frac{%
q_{n_{1}}}{p_{n_{1}}}\right) \frac{1}{[n_{1}+1]_{p_{n_{1}},q_{n_{1}}}}%
+(p_{n_{1}}^{n_{1}})^{2}\left( 1+\frac{q_{n_{1}}}{p_{n_{1}}}\right) \frac{1}{%
[n_{1}+1]_{p_{n_{1}},q_{n_{1}}}^{2}}-1\right\} $\newline
$+p_{n_{1}}^{n_{1}}.\frac{p_{n_{1}}}{q_{n_{1}}}\left( \frac{1}{%
[n_{1}+1]_{p_{n_{1}},q_{n_{1}}}}-p_{n_{1}}^{n_{1}}\frac{1}{%
[n_{1}+1]_{p_{n_{1}},q_{n_{1}}}^{2}}\right) .$\newline
This implies that the condition (3.3) holds for also $i=2,~~~j=0$ and also
in similar way for $i=0,~~~j=2$ the proof is true. This completes the proof.
\end{proof}

\section{Rate of Convergence}

In this section, rate of convergence of operators (2.1) by means of modulus
of continuity of some bivariate modulus of smoothness functions are
introduced.\newline
The modulus of continuity for bivariate case is defined as follows. For $%
f\in H_{\omega }(\mathbb{R}_{+}^{2})$ is defined by\newline
$\widetilde{\omega }(f;\delta _{1},\delta _{2})=\sup_{u,x\geq 0}\left\{ %
\bigg{|}f(u,v)-f(x,y)\bigg{|};~~\bigg{|}\frac{u}{1+u}-\frac{x}{1+x}\bigg{|}%
\leq \delta _{1},\bigg{|}\frac{v}{1+v}-\frac{y}{1+y}\bigg{|}\leq \delta
_{2},~~(u,v)\in \mathbb{R}_{+}^{2},~~(x,y)\in \mathbb{R}_{+}^{2}\right\} .$%
\newline
where for all $f\in H_{\omega }(\mathbb{R}_{+})$ $\widetilde{\omega }%
(f;\delta _{1},\delta _{2})$ satisfies the following conditions:

\begin{enumerate}
\item[$($i$)$] $\lim_{\delta_1,\delta_2 \to 0}\widetilde{\omega}(f;
\delta_1,\delta_2)\to 0$

\item[$($ii$)$] $\mid f(u,v)-f(x,y) \mid \leq \widetilde{\omega}(f;
\delta_1,\delta_2) \left(\frac{\mid \frac{u}{1+u}-\frac{x}{1+x}\mid}{\delta_1%
}+1 \right)\left(\frac{\mid \frac{v}{1+v}-\frac{y}{1+y}\mid}{\delta_2}+1
\right).$
\end{enumerate}

\begin{thm}
Let $p=p_{n_1},~~p=p_{n_2}$ and $q=q_{n_1},~~~q=q_{n_2}$ satisfying (3.1),
for $0<q_{n_1}<p_{n_1}\leq 1,~~0<q_{n_2}<p_{n_2}\leq 1$ and if $%
L^{(p_{n_1},p_{n_2}),(q_{n_1},q_{n_2})}_{n_1,n_2}$ be the sequence of a
linear positive operator defined in (2.1). Then for all $f \in H_\omega(%
\mathbb{R}_+^2)$ and $x,y \geq 0$, we have

\begin{equation*}
\mid L^{(p_{n_1},p_{n_2}),(q_{n_1},q_{n_2})}_{n_1,n_2} (f;x,y)-f(x,y)\mid
\leq 4 \omega (f;\delta_{n_1}(x), \delta_{n_2}(y))~~~~~~~~~~~~~~~~\eqno(4.1)
\end{equation*}
where
\begin{equation*}
\delta_{n_1}(x)=\frac{x^2}{(1+x)^2}\left( \frac{%
p_{n_1}q_{n_1}^2[n_1]_{p_{n_1},q_{n_1}}[n_1-1]_{p_{n_1},q_{n_1}}}{%
[n_1+1]_{p_{n_1},q_{n_1}}^2} \frac{1+x}{p_{n_1}+q_{n_1}x} -2\frac{%
p_{n_1}[n_1]_{p_{n_1},q_{n_1}}}{[n_1+1]_{p_{n_1},q_{n_1}}}+1\right) +\frac{%
p_{n_1}^{n_1+1}[n_1]_{p_{n_1},q_{n_1}}}{[n_1+1]_{p_{n_1},q_{n_1}}^2}\frac{x}{%
1+x}
\end{equation*}
\begin{equation*}
\delta_{n_2}(y)=\frac{y^2}{(1+y)^2}\left( \frac{%
p_{n_2}q_{n_2}^2[n_2]_{p_{n_2},q_{n_2}}[n_2-1]_{p_{n_2},q_{n_2}}}{%
[n_2+1]_{p_{n_2},q_{n_2}}^2} \frac{1+y}{p_{n_2}+q_{n_2}y} -2\frac{%
p_{n_2}[n_2]_{p_{n_2},q_{n_2}}}{[n_2+1]_{p_{n_2},q_{n_2}}}+1\right) +\frac{%
p_{n_2}^{n_2+1}[n_2]_{p_{n_2},q_{n_2}}}{[n_2+1]_{p_{n_2},q_{n_2}}^2}\frac{y}{%
1+y}.
\end{equation*}
\end{thm}

\begin{proof}
\begin{eqnarray*}
\mid L^{(p_{n_1},p_{n_2}),(q_{n_1},q_{n_2})}_{n_1,n_2} (f;x,y)-f(x,y) \mid
&\leq & L^{(p_{n_1},p_{n_2}),(q_{n_1},q_{n_2})}_{n_1,n_2}\left(\mid
f(u,v)-f(x,y) \mid;x,y \right) \\
&\leq & \widetilde{\omega}(f; \delta_{n_1}, \delta_{n_2}) \left\{1+\frac{1}{%
\delta_{n_1}} L^{(p_{n_1},p_{n_2}),(q_{n_1},q_{n_2})}_{n_1,n_2} \left( %
\bigg{|} \frac{u}{1+u}-\frac{x}{1+x}\big{|};x,y \right)\right\} \\
&\times & \left\{1+\frac{1}{\delta_{n_2}}
L^{(p_{n_1},p_{n_2}),(q_{n_1},q_{n_2})}_{n_1,n_2} \left( \bigg{|} \frac{v}{%
1+v}-\frac{y}{1+y}\big{|};x,y \right)\right\}.
\end{eqnarray*}
Now by using the Cauchy-Schwarz inequality, we have\newline
\newline
$\mid L^{(p_{n_1},p_{n_2}),(q_{n_1},q_{n_2})}_{n_1,n_2} (f;x,y)-f(x,y) \mid$
\begin{eqnarray*}
&\leq & \widetilde{\omega}(f; \delta_{n_1}, \delta_{n_2}) \left\{1+\frac{1}{%
\delta_{n_1}} \left[ \left(
L^{(p_{n_1},p_{n_2}),(q_{n_1},q_{n_2})}_{n_1,n_2} \left(\frac{u}{1+u}-\frac{x%
}{1+x}\right)^2;x,y \right) \right]^{\frac{1}{2}} \left(
L^{(p_{n_1},p_{n_2}),(q_{n_1},q_{n_2})}_{n_1,n_2}(e_{00};x,y)\right)^{\frac{1%
}{2}}\right\} \\
&\times & \widetilde{\omega}(f; \delta_{n_1}, \delta_{n_2}) \left\{1+\frac{1%
}{\delta_{n_2}} \left[ \left(
L^{(p_{n_1},p_{n_2}),(q_{n_1},q_{n_2})}_{n_1,n_2} \left(\frac{v}{1+v}-\frac{y%
}{1+y}\right)^2;x,y \right) \right]^{\frac{1}{2}} \left(
L^{(p_{n_1},p_{n_2}),(q_{n_1},q_{n_2})}_{n_1,n_2}(e_{00};x,y)\right)^{\frac{1%
}{2}}\right\}
\end{eqnarray*}
$\leq \widetilde{\omega}(f; \delta_{n_1}, \delta_{n_2}) \left\{ 1+ \frac{1}{%
\delta_{n_1}} \left[ \frac{x^2}{(1+x)^2}\left( \frac{%
p_{n_1}q_{n_1}^2[n_1]_{p_{n_1},q_{n_1}}[n_1-1]_{p_{n_1},q_{n_1}}}{%
[n_1+1]_{p_{n_1},q_{n_1}}^2} \frac{1+x}{p_{n_1}+q_{n_1}x} -2\frac{%
p_{n_1}[n_1]_{p_{n_1},q_{n_1}}}{[n_1+1]_{p_{n_1},q_{n_1}}}+1\right) +\frac{%
p_{n_1}^{n_1+1}[n_1]_{p_{n_1},q_{n_1}}}{[n_1+1]_{p_{n_1},q_{n_1}}^2}\frac{x}{%
1+x}\right]^{\frac{1}{2}}\right\}$\newline
$\times \left\{ 1+ \frac{1}{\delta_{n_2}} \left[ \frac{y^2}{(1+y)^2}\left(
\frac{p_{n_2}q_{n_2}^2[n_2]_{p_{n_2},q_{n_2}}[n_2-1]_{p_{n_2},q_{n_2}}}{%
[n_2+1]_{p_{n_2},q_{n_2}}^2} \frac{1+y}{p_{n_2}+q_{n_2}y} -2\frac{%
p_{n_2}[n_2]_{p_{n_2},q_{n_2}}}{[n_2+1]_{p_{n_2},q_{n_2}}}+1\right) +\frac{%
p_{n_2}^{n_2+1}[n_2]_{p_{n_2},q_{n_2}}}{[n_2+1]_{p_{n_2},q_{n_2}}^2}\frac{y}{%
1+y}\right]^{\frac{1}{2}}\right\}$.\newline
This completes the proof.
\end{proof}

Now we will give an estimate concerning the rate of convergence by means of
Lipschtz type maximal functions. In \cite{ern,biv}, the Lipschtz type
maximal function space on $E \times E \subset \mathbb{R}_+ \times \mathbb{R}%
_+$ is defined as follows:\newline

\begin{equation*}
\widetilde{W}_{\alpha_1,\alpha_2 E^2}=\{ f:\sup (1+u)^{\alpha_1}
(1+v)^{\alpha_2} \left({f}_{\alpha_1, \alpha_2} (u,v)-{f}_{\alpha_1,
\alpha_2} (x,y)\right)
\end{equation*}
\begin{equation*}
\leq {M} \frac{1}{(1+x)^{\alpha_1}}\frac{1}{(1+y)^{\alpha_2}}: x,y \leq 0,~%
\mbox{and}~ (u,v)\in E^2\}
\end{equation*}
where $f$ is bounded and continuous function on $\mathbb{R}_{+}$, ${M}$ is a
positive constant. For $0\leq \alpha _{1}\leq 1,~~0\leq \alpha _{2}\leq 1$,
we define the function $f_{\alpha _{1},\alpha _{2}}$ as follows:
\begin{equation*}
{f}_{\alpha _{1},\alpha _{2}}(u,v)-{f}_{\alpha _{1},\alpha _{2}}(x,y)=\frac{%
\mid f(u,v)-f(x,y)\mid }{\mid u-x\mid ^{\alpha _{1}}\mid v-y\mid ^{\alpha
_{2}}}
\end{equation*}%
Also, let $d(x,E)$ be the distance between $x$ and $E$, that is
\begin{equation*}
d(x,E)=\inf \{\mid x-y\mid ;y\in E\}.
\end{equation*}

\begin{thm}
Let $0 < \alpha_1 \leq 1,~~0< \alpha_2 \leq 1$, then for all $f \in
\widetilde{W}_{\alpha_1, \alpha_2, E^2}$ we have\newline
\newline
$\mid L^{(p_{n_1},p_{n_2}),(q_{n_1},q_{n_2})}_{n_1,n_2} (f;x,y) -f(x,y) \mid$%
\newline
\begin{equation*}
\leq {M}\left( \delta_{n_1}^{\frac{\alpha_1}{2}}(x)\delta_{n_2}^{\frac{%
\alpha_2}{2}}(y) + \delta_{n_1}^{\frac{\alpha_1}{2}}(x)\left(
d(x,E)\right)^{\alpha_1} + \delta_{n_2}^{\frac{\alpha_2}{2}}(y)\left(
d(y,E)\right)^{\alpha_2} +2 \left( d(x,E)\right)^{\alpha_1} \left(
d(y,E)\right)^{\alpha_2}\right)
\end{equation*}
where $\delta_{n_1}(x)$ and $\delta_{n_2}(y)$ are defined in Theorem 4.1.
\end{thm}

\begin{proof}
Let $x,y \in \mathbb{R}_+$, then there exits a $(x_0,y_0) \in {E} \times E$
such that $\mid x-x_0 \mid = d(x,E)$, and $\mid y-y_0 \mid = d(y,E)$. Thus
we can write
\begin{equation*}
\mid f(u,v) -f(x,y) \mid \leq \mid f(u,v)-f(x_0,y_0) \mid + \mid
f(x_0,y_0)-f(x,y) \mid.
\end{equation*}

Since $L^{(p_{n_1},p_{n_2}),(q_{n_1},q_{n_2})}_{n_1,n_2}$ is a positive
linear operator, $f \in \widetilde{W}_{\alpha_1, \alpha_2, E^2}$ and by
using the previous inequality we have\newline
\newline
$\mid L^{(p_{n_1},p_{n_2}),(q_{n_1},q_{n_2})}_{n_1,n_2}(f;x,y) -f(x,y) \mid $%
\newline
$\leq \mid L^{(p_{n_1},p_{n_2}),(q_{n_1},q_{n_2})}_{n_1,n_2}(\mid
f(u,v)-f(x_0,y_0) \mid ;x,y)+ \mid f(x_0,y_0)-f(x,y) \mid
L^{(p_{n_1},p_{n_2}),(q_{n_1},q_{n_2})}_{n_1,n_2}(e_{00};x,y)$\newline

$\leq {M} L^{(p_{n_1},p_{n_2}),(q_{n_1},q_{n_2})}_{n_1,n_2}\left( \bigg{|}
\frac{u}{1+u}-\frac{x_0}{1+x_0} \bigg{|}^{\alpha_1} \bigg{|} \frac{v}{1+v}-%
\frac{y_0}{1+y_0} \bigg{|}^{\alpha_2} ;x,y \right)$\newline
$+{M} \bigg{|} \frac{x}{1+x}-\frac{x_0}{1+x_0} \bigg{|}^{\alpha_1} \bigg{|}
\frac{y}{1+y}-\frac{y_0}{1+y_0} \bigg{|}^{\alpha_2}
L^{(p_{n_1},p_{n_2}),(q_{n_1},q_{n_2})}_{n_1,n_2}(e_{00};x,y) $\newline

Since $(a+b)^\alpha\leq a^\alpha +b^\alpha$, for $a,b \geq 0$ and $0 \leq
\alpha \leq 1$, which consequently imply

\begin{equation*}
\bigg{|} \frac{u}{1+u}-\frac{x_0}{1+x_0} \bigg{|}^{\alpha_1} \leq \bigg{|}
\frac{u}{1+u}-\frac{x}{1+x} \bigg{|}^{\alpha_1} + \bigg{|} \frac{x}{1+x}-%
\frac{x_0}{1+x_0} \bigg{|}^{\alpha_1}
\end{equation*}
\begin{equation*}
\bigg{|} \frac{v}{1+v}-\frac{y_0}{1+y_0} \bigg{|}^{\alpha_2} \leq \bigg{|}
\frac{v}{1+v}-\frac{y}{1+y} \bigg{|}^{\alpha_2} + \bigg{|} \frac{y}{1+y}-%
\frac{y_0}{1+y_0} \bigg{|}^{\alpha_2}
\end{equation*}
\begin{eqnarray*}
\mid L^{(p_{n_1},p_{n_2}),(q_{n_1},q_{n_2})}_{n_1,n_2}(f;x,y) -f(x,y) \mid &
\leq & L^{(p_{n_1},p_{n_2}),(q_{n_1},q_{n_2})}_{n_1,n_2}\left( \bigg{|}
\frac{u}{1+u}-\frac{x}{1+x} \bigg{|}^{\alpha_1} \bigg{|} \frac{v}{1+v}-\frac{%
y}{1+y} \bigg{|}^{\alpha_2} ;x,y \right) \\
&+& \bigg{|} \frac{y}{1+y}-\frac{y_0}{1+y_0} \bigg{|}^{\alpha_2}
L^{(p_{n_1},p_{n_2}),(q_{n_1},q_{n_2})}_{n_1,n_2}\left( \bigg{|} \frac{u}{1+u%
}-\frac{x}{1+x} \bigg{|}^{\alpha_1} ;x,y \right) \\
&+& \bigg{|} \frac{x}{1+x}-\frac{x_0}{1+x_0} \bigg{|}^{\alpha_1}
L^{(p_{n_1},p_{n_2}),(q_{n_1},q_{n_2})}_{n_1,n_2}\left( \bigg{|} \frac{v}{1+v%
}-\frac{y}{1+y} \bigg{|}^{\alpha_2} ;x,y \right) \\
&+& \bigg{|} \frac{x}{1+x}-\frac{x_0}{1+x_0} \bigg{|}^{\alpha_1} \bigg{|}
\frac{y}{1+y}-\frac{y_0}{1+y_0} \bigg{|}^{\alpha_2}
L^{(p_{n_1},p_{n_2}),(q_{n_1},q_{n_2})}_{n_1,n_2}(e_{00};x,y).
\end{eqnarray*}
By applying the Hölder inequality with $p_1=\frac{2}{\alpha_1},~~p_2= \frac{2%
}{\alpha_2}$ and $q_1 =\frac{2}{2-\alpha_1},~~q_2 =\frac{2}{2-\alpha_2}$,
then we have

$L^{(p_{n_1},p_{n_2}),(q_{n_1},q_{n_2})}_{n_1,n_2}\left( \bigg{|} \frac{u}{%
1+u}-\frac{x}{1+x} \bigg{|}^{\alpha_1} \bigg{|} \frac{v}{1+v}-\frac{y}{1+y} %
\bigg{|}^{\alpha_2} ;x,y \right)$\newline
$= L^{(p_{n_1},p_{n_2}),(q_{n_1},q_{n_2})}_{n_1,n_2}\left( \bigg{|} \frac{u}{%
1+u}-\frac{x}{1+x} \bigg{|}^{\alpha_1};x,y \right)
L^{(p_{n_1},p_{n_2}),(q_{n_1},q_{n_2})}_{n_1,n_2}\left( \bigg{|}\frac{v}{1+v}%
-\frac{y}{1+y} \bigg{|}^{\alpha_2} ;x,y \right)$\newline
$\leq \left(L^{(p_{n_1},p_{n_2}),(q_{n_1},q_{n_2})}_{n_1,n_2}\left( \bigg{|}
\frac{u}{1+u}-\frac{x}{1+x} \bigg{|}^{2};x,y \right)\right)^{\frac{\alpha_1}{%
2}} \left(
L^{(p_{n_1},p_{n_2}),(q_{n_1},q_{n_2})}_{n_1,n_2}(e_{00};x,y)\right)^{\frac{%
2-\alpha_1}{2}}$\newline
$\times \left(L^{(p_{n_1},p_{n_2}),(q_{n_1},q_{n_2})}_{n_1,n_2}\left( %
\bigg{|} \frac{v}{1+v}-\frac{y}{1+y} \bigg{|}^{2};x,y \right)\right)^{\frac{%
\alpha_2}{2}} \left(
L^{(p_{n_1},p_{n_2}),(q_{n_1},q_{n_2})}_{n_1,n_2}(e_{00};x,y) \right)^{\frac{%
2-\alpha_2}{2}}$.\newline
This completes the proof.
\end{proof}

\begin{cor}
If it is taken $E= \mathbb{R}_+$ as a special case of Theorem 4.2, then for
all $f \in \widetilde{W}_{\alpha_1, \alpha_2, \mathbb{R}_+^2}$, we have
\begin{equation*}
\mid L^{(p_{n_1},p_{n_2}),(q_{n_1},q_{n_2})}_{n_1,n_2} (f;x,y) -f(x,y) \mid
\leq \mathcal{M} \left(\delta_{n_1}^{\frac{\alpha_1}{2}}(x)\delta_{n_2}^{%
\frac{\alpha_2}{2}}(y)\right),
\end{equation*}
\end{cor}
where $\delta_{n_1}(x)$ and $\delta_{n_2}(y)$ are defined in Theorem 4.1.

\section{\textbf{{\ Some Generalization of $%
L^{(p_{n_1},p_{n_2}),(q_{n_1},q_{n_2})}_{n_1,n_2}$}}}

In this section, similar as in \cite{aral1,n1}, we define some
generalization of the operators (2.1) for $(p,q)$-integers as follows:%
\newline
\newline
$L_{n_{1},n_{2}}^{(p_{n_{1}},p_{n_{2}}),(q_{n_{1}},q_{n_{2}}),(\gamma
_{1},\gamma _{2})}(f;x,y)$\newline
\begin{equation*}
=\frac{1}{\ell _{n_{1}}^{(p_{n_{1}},q_{n_{1}})}(x)}\frac{1}{\ell
_{n_{2}}^{(p_{n_{2}},q_{n_{2}})}(y)}\sum_{k_{1}=0}^{n_{1}}%
\sum_{k_{2}=0}^{n_{2}}f\left( \frac{%
p_{n_{1}}^{n_{1}+1-k_{1}}[k_{1}]_{p_{n_{1}},q_{n_{1}}}+\gamma _{1}}{%
b_{n_{1},k_{1}}},\frac{%
p_{n_{2}}^{n_{2}+1-k_{2}}[k_{2}]_{p_{n_{2}},q_{n_{2}}}+\gamma _{2}}{%
b_{n_{2},k_{2}}}\right)
\end{equation*}%
\begin{equation*}
\times p_{n_{1}}^{\frac{(n_{1}-k_{1})(n_{1}-k_{1}-1)}{2}}q_{n_{1}}^{\frac{%
k_{1}(k_{1}-1)}{2}}p_{n_{2}}^{\frac{(n_{2}-k_{2})(n_{2}-k_{2}-1)}{2}%
}q_{n_{2}}^{\frac{k_{2}(k_{2}-1)}{2}}\left[
\begin{array}{c}
n_{1} \\
k_{1}%
\end{array}%
\right] _{p_{n_{1}},q_{n_{1}}}\left[
\begin{array}{c}
n_{2} \\
k_{2}%
\end{array}%
\right] _{p_{n_{2}},q_{n_{2}}}x^{k_{1}}y^{k_{2}},~~~~~~(\gamma _{1},\gamma
_{2}\in \mathbb{R})~~~~~~~~~~~~~~~~~~~~~~~\eqno(5.1)
\end{equation*}
where $b_{n_i,k_i}$ for $i=1,2$ satisfies the following conditions:
\begin{equation*}
p_{n_i}^{n_i-k_i+1}[k_i]_{p_{n_i},q_{n_i}}+b_{n_i,k_i}=c_{n_i}~~~~~~%
\mbox{and}~~~~ \frac{[n_i]_{p_{n_i},q_{n_i}}}{c_{n_i}}\to 1,~~~~~~\mbox{for}%
~~~~n_i \to \infty.
\end{equation*}
It is easy to check that if $%
b_{n_i,k_i}=q_{n_i}^{k_i}[n_i-k_i+1]_{p_{n_i},q_{n_i}}+\beta_i$ for any $%
n_i,k_i$ and $0<q_{n_i}<p_{n_i}\leq 1$, then $%
c_{n_i}=[n_i+1]_{p_{n_i},q_{n_i}}+\beta_i$. If we choose $%
p_{n_i}=1,~~\gamma_i=0$, for $i=1,2$ then operators reduces to $q$%
-bivariate-BBH opeartors defined in \cite{ern,biv}.

\begin{thm}
Let $p=p_{n_1},~~p=p_{n_2}$ and $q=q_{n_1},~~q=q_{n_2}$ satisfying (3.1),
for $0<q_{n_1}<p_{n_1}\leq 1,~~0<q_{n_2}<p_{n_2}\leq 1$ and if $%
L^{(p_{n_1},p_{n_2}),(q_{n_1},q_{n_2}),(\gamma_1,\gamma_2)}_{n_1,n_2}$ is
defined by (5.1), then for any function $f \in \widetilde{W}_{\alpha_1,
\alpha_2, E^2}$, we have\newline

$\lim_{n_1,n_2} \parallel
L^{(p_{n_1},p_{n_2}),(q_{n_1},q_{n_2}),(\gamma_1,%
\gamma_2)}_{n_1,n_2}(f;x,y)-f(x,y) \parallel_{C_{B}(\mathbb{R}^2_+)} $%
\newline
$\leq 3 M \max \left( \frac{[n_1]_{p_{n_1},q_{n_1}}}{c_{n_1} +\gamma_1}%
\right)^{\alpha_1} \left( \frac{\gamma_1}{[n_1]_{p_{n_1},q_{n_1}}}%
\right)^{\alpha_1} \left( \frac{[n_2]_{p_{n_2},q_{n_2}}}{c_{n_2} +\gamma_2}%
\right)^{\alpha_2} \left( \frac{\gamma_2}{[n_2]_{p_{n_2},q_{n_2}}}%
\right)^{\alpha_2},$\newline

$\bigg{|} 1- \frac{[n_1+1]_{p_{n_1},q_{n_1}}}{c_{n_1}+\gamma_1}\bigg{|}%
^{\alpha_1} \left( \frac{p_{n_1}[n_1]_{p_{n_1},q_{n_1}}}{%
[n_1+1]_{p_{n_1},q_{n_1}}} \right)^{\alpha_1} \bigg{|} 1- \frac{%
[n_2+1]_{p_{n_2},q_{n_2}}}{c_{n_2}+\gamma_2}\bigg{|}^{\alpha_2} \left( \frac{%
p_{n_2}[n_2]_{p_{n_2},q_{n_2}}}{[n_2+1]_{p_{n_2},q_{n_2}}}
\right)^{\alpha_2},$\newline

$\left(1- 2\frac{p_{n_1}[n_1]_{p_{n_1},q_{n_1}}}{[n_1+1]_{p_{n_1},q_{n_1}}}+
\frac{[n_1]_{p_{n_1},q_{n_1}}[n_1-1]_{p_{n_1},q_{n_1}}}{%
[n_1+1]_{p_{n_1},q_{n_1}}^2}q_{n_1}\right) \left(1- 2\frac{%
p_{n_2}[n_2]_{p_{n_2},q_{n_2}}}{[n_2+1]_{p_{n_2},q_{n_2}}}+ \frac{%
[n_2]_{p_{n_2},q_{n_2}}[n_2-1]_{p_{n_2},q_{n_2}}}{[n_2+1]_{p_{n_2},q_{n_2}}^2%
}q_{n_2}\right).$
\end{thm}

\begin{proof}
Using (2.1) and (5.1) we have\newline

$\bigg{|} L^{(p_{n_1},p_{n_2}),(q_{n_1},q_{n_2}),(\gamma_1,%
\gamma_2)}_{n_1,n_2}(f;x,y)-f(x,y) \bigg{|} \leq \frac{1}{%
\ell^{(p_{n_1},q_{n_1})}_{n_1}(x)} \frac{1}{\ell^{(p_{n_2},q_{n_2})}_{n_2}(y)%
} \sum_{k_1=0}^{n_1} \sum_{k_2=0}^{n_2}$\newline

$\times \bigg{|} f \left(\frac{p_{n_1}^{n_1-k_1+1}[k_1]_{p_{n_1},q_{n_1}}+%
\gamma_1}{b_{n_1,k_1}}, \frac{p_{n_2}^{n_2-k_2+1}[k_2]_{p_{n_2},q_{n_2}}+%
\gamma_2}{b_{n_2,k_2}} \right) -f\left(\frac{%
p_{n_1}^{n_1-k_1+1}[k_1]_{p_{n_1},q_{n_1}}}{\gamma_1+b_{n_1,k_1}}, \frac{%
p_{n_2}^{n_2-k_2+1}[k_2]_{p_{n_2},q_{n_2}}}{\gamma_2+b_{n_2,k_2}} \right)%
\bigg{|} $\newline

$\times p_{n_1}^{\frac{(n_1-k_1)(n_1-k_1-1)}{2}}q_{n_1}^{\frac{k_1(k_1-1)}{2}%
} p_{n_2}^{\frac{(n_2-k_2)(n_2-k_2-1)}{2}}q_{n_2}^{\frac{k_2(k_2-1)}{2}} %
\left[
\begin{array}{c}
n_1 \\
k_1%
\end{array}%
\right] _{p_{n_1},q_{n_1}} \left[
\begin{array}{c}
n_2 \\
k_2%
\end{array}%
\right] _{p_{n_2},q_{n_2}} x^{k_1}y^{k_2}$\newline

$+ \frac{1}{\ell^{(p_{n_1},q_{n_1})}_{n_1}(x)} \frac{1}{%
\ell^{(p_{n_2},q_{n_2})}_{n_2}(y)}\sum_{k_1=0}^{n_1}\sum_{k_2=0}^{n_2}$%
\newline

$\times \bigg{|} f\left(\frac{p_{n_1}^{n_1-k_1+1}[k_1]_{p_{n_1},q_{n_1}}}{%
\gamma_1+b_{n_1,k_1}}, \frac{p_{n_2}^{n_2-k_2+1}[k_2]_{p_{n_2},q_{n_2}}}{%
\gamma_2+b_{n_2,k_2}} \right) -f\left(\frac{%
p_{n_1}^{n_1-k_1+1}[k_1]_{p_{n_1},q_{n_1}}}{%
[n_1-k_1+1]_{p_{n_1},q_{n_1}}q_{n_1}^{k_1}}, \frac{%
p_{n_2}^{n_2-k_2+1}[k_2]_{p_{n_2},q_{n_2}}}{%
[n_2-k_2+1]_{p_{n_2},q_{n_2}}q_{n_2}^{k_2}} \right)\bigg{|} $\newline

$\times p_{n_1}^{\frac{(n_1-k_1)(n_1-k_1-1)}{2}}q_{n_1}^{\frac{k_1(k_1-1)}{2}%
} p_{n_2}^{\frac{(n_2-k_2)(n_2-k_2-1)}{2}}q_{n_2}^{\frac{k_2(k_2-1)}{2}} %
\left[
\begin{array}{c}
n_1 \\
k_1%
\end{array}%
\right] _{p_{n_1},q_{n_1}} \left[
\begin{array}{c}
n_2 \\
k_2%
\end{array}%
\right] _{p_{n_2},q_{n_2}} x^{k_1} y^{k_2}$\newline

$+ \mid L^{(p_{n_1},p_{n_2}),(q_{n_1},q_{n_2})}_{n_1,n_2}(f;x,y)-f(x,y)
\mid. $\newline

Since $f\in \widetilde{W}_{\alpha _{1},\alpha _{2},E^{2}}$ and by using the
Corollary 4.3 we can write\newline

$\bigg{|} L^{(p_{n_1},p_{n_2}),(q_{n_1},q_{n_2}),(\gamma_1,%
\gamma_2)}_{n_1,n_2}(f;x,y)-f(x,y) \bigg{|}$\newline
$\leq M \frac{1}{\ell^{(p_{n_1},q_{n_1})}_{n_1}(x)} \frac{1}{%
\ell^{(p_{n_2},q_{n_2})}_{n_2}(y)} \sum_{k_1=0}^{n_1}\sum_{k_2=0}^{n_2} %
\bigg{|} \frac{p_{n_1}^{n_1-k_1+1}[k_1]_{p_{n_1},q_{n_1}}+\gamma_1}{%
p_{n_1}^{n_1-k_1+1}[k_1]_{p_{n_1},q_{n_1}}+\gamma_1+b_{n_1,k_1}} -\frac{%
p_{n_1}^{n_1-k_1+1}[k_1]_{p_{n_1},q_{n_1}}}{%
\gamma_1+p_{n_1}^{n_1-k_1+1}[k_1]_{p_{n_1},q_{n_1}}+b_{n_1,k_1}}\bigg{|}%
^{\alpha_1} $\newline

$\times \bigg{|} \frac{p_{n_2}^{n_2-k_2+1}[k_2]_{p_{n_2},q_{n_2}}+\gamma_2}{%
p_{n_2}^{n_2-k_2+1}[k_2]_{p_{n_2},q_{n_2}}+\gamma_2+b_{n_2,k_2}} -\frac{%
p_{n_2}^{n_2-k_2+1}[k_2]_{p_{n_2},q_{n_2}}}{%
\gamma_2+p_{n_2}^{n_2-k_2+1}[k_2]_{p_{n_2},q_{n_2}}+b_{n_2,k_2}}\bigg{|}%
^{\alpha_2}$\newline

$\times p_{n_1}^{\frac{(n_1-k_1)(n_1-k_1-1)}{2}}q_{n_1}^{\frac{k_1(k_1-1)}{2}%
}p_{n_2}^{\frac{(n_2-k_2)(n_2-k_2-1)}{2}}q_{n_2}^{\frac{k_2(k_2-1)}{2}} %
\left[
\begin{array}{c}
n_1 \\
k_1%
\end{array}%
\right] _{p_{n_1},q_{n_1}} \left[
\begin{array}{c}
n_2 \\
k_2%
\end{array}%
\right] _{p_{n_2},q_{n_2}} x^{k_1} y^{k_2}$\newline

$+ M \frac{1}{\ell^{(p_{n_1},q_{n_1})}_{n_1}(x)} \frac{1}{%
\ell^{(p_{n_2},q_{n_2})}_{n_2}(y)} \sum_{k_1=0}^{n_1}\sum_{k_2=0}^{n_2} %
\bigg{|} \frac{p_{n_1}^{n_1-k_1+1}[k_1]_{p_{n_1},q_{n_1}}}{%
p_{n_1}^{n_1-k_1+1}[k_1]_{p_{n_1},q_{n_1}}+\gamma_1+ b_{n_1,k_1}} -\frac{%
p_{n_1}^{n_1-k_1+1}[k_1]_{p_{n_1},q_{n_1}}}{%
p_{n_1}^{n_1-k_1+1}[k_1]_{p_{n_1},q_{n_1}}+[n_1-k_1+1]_{p_{n_1},q_{n_1}}q_{n_1}^{k_1}%
}\bigg{|}$\newline

$\times \bigg{|} \frac{p_{n_2}^{n_2-k_2+1}[k_2]_{p_{n_2},q_{n_2}}}{%
p_{n_2}^{n_2-k_2+1}[k_2]_{p_{n_2},q_{n_2}}+\gamma_2+ b_{n_2,k_2}} -\frac{%
p_{n_2}^{n_2-k_2+1}[k_2]_{p_{n_2},q_{n_2}}}{%
p_{n_2}^{n_2-k_2+1}[k_2]_{p_{n_2},q_{n_2}}+[n_2-k_2+1]_{p_{n_2},q_{n_2}}q_{n_2}^{k_2}%
}\bigg{|}$\newline

$\times p_{n_1}^{\frac{(n_1-k_1)(n_1-k_1-1)}{2}}q_{n_1}^{\frac{k_1(k_1-1)}{2}%
} p_{n_2}^{\frac{(n_2-k_2)(n_2-k_2-1)}{2}}q_{n_2}^{\frac{k_2(k_2-1)}{2}} %
\left[
\begin{array}{c}
n_1 \\
k_1%
\end{array}%
\right] _{p_{n_1},q_{n_1}} \left[
\begin{array}{c}
n_2 \\
k_2%
\end{array}%
\right] _{p_{n_2},q_{n_2}} x^{k_1}y^{k_2} $\newline
$+ M \delta_{n_1}^{\frac{\alpha_1}{2}}(x) \delta_{n_2}^{\frac{\alpha_2}{2}%
}(y).$\newline

This implies that\newline

$\bigg{|} L^{(p_{n_1},p_{n_2}),(q_{n_1},q_{n_2}),(\gamma_1,%
\gamma_2)}_{n_1,n_2}(f;x,y)-f(x,y) \bigg{|}$\newline

$\leq M \left( \frac{[n_1]_{p_{n_1},q_{n_1}}}{c_{n_1} +\gamma_1}%
\right)^{\alpha_1}\left( \frac{\gamma_1}{[n_1]_{p_{n_1},q_{n_1}}}%
\right)^{\alpha_1} \left( \frac{[n_2]_{p_{n_2},q_{n_2}}}{c_{n_2} +\gamma_2}%
\right)^{\alpha_2}\left( \frac{\gamma_2}{[n_2]_{p_{n_2},q_{n_2}}}%
\right)^{\alpha_2} $\newline

$+M \frac{1}{\ell^{(p_{n_1},q_{n_1})}_{n_1}(x)} \frac{1}{%
\ell^{(p_{n_2},q_{n_2})}_{n_2}(y)} \bigg{|} 1- \frac{%
[n_1+1]_{p_{n_1},q_{n_1}}}{c_{n_1}+\gamma_1} \bigg{|}^{\alpha_1} \bigg{|} 1-
\frac{[n_2+1]_{p_{n_2},q_{n_2}}}{c_{n_2}+\gamma_2} \bigg{|}^{\alpha_2}$%
\newline
$\times \sum_{k_1=0}^{n_1} \sum_{k_2=0}^{n_2} \left(\frac{%
p_{n_1}^{n_1-k_1+1}[k_1]_{p_{n_1},q_{n_1}}}{[n_1+1]_{p_{n_1},q_{n_1}}}
\right)^{\alpha_1} \left(\frac{p_{n_2}^{n_2-k_2+1}[k_2]_{p_{n_2},q_{n_2}}}{%
[n_2+1]_{p_{n_2},q_{n_2}}} \right)^{\alpha_2} $\newline

$\times p_{n_1}^{\frac{(n_1-k_1)(n_1-k_1-1)}{2}}q_{n_1}^{\frac{k_1(k_1-1)}{2}%
} p_{n_2}^{\frac{(n_2-k_2)(n_2-k_2-1)}{2}}q_{n_2}^{\frac{k_2(k_2-1)}{2}} %
\left[
\begin{array}{c}
n_1 \\
k_1%
\end{array}%
\right] _{p_{n_1},q_{n_1}} \left[
\begin{array}{c}
n_2 \\
k_2%
\end{array}%
\right] _{p_{n_2},q_{n_2}} x^{k_1}y^{k_2} + M \delta_{n_1}^{\frac{\alpha_1}{2%
}}(x) \delta_{n_2}^{\frac{\alpha_2}{2}}(y)$\newline

$= M \left( \frac{[n_1]_{p_{n_1},q_{n_1}}}{c_{n_1} +\gamma_1}%
\right)^{\alpha_1}\left( \frac{\gamma_1}{[n_1]_{p_{n_1},q_{n_1}}}%
\right)^{\alpha_1} \left( \frac{[n_2]_{p_{n_2},q_{n_2}}}{c_{n_2} +\gamma_2}%
\right)^{\alpha_2}\left( \frac{\gamma_2}{[n_2]_{p_{n_2},q_{n_2}}}%
\right)^{\alpha_2}$\newline

$+M \bigg{|} 1- \frac{[n_1+1]_{p_{n_1},q_{n_1}}}{c_{n_1}+\gamma_1} \bigg{|}%
^{\alpha_1} \bigg{|} 1- \frac{[n_2+1]_{p_{n_2},q_{n_2}}}{c_{n_2}+\gamma_2} %
\bigg{|}^{\alpha_2} L_{p_{n_1},q_{n_1}, p_{n_2},q_{n_2}}^{n_1,n_2}\left(
\left(\frac{u}{1+u}\right)^{\alpha_1};x,y \right) L_{p_{n_1},q_{n_1} ,
p_{n_2},q_{n_2}}^{n_1, n_2}\left( \left(\frac{v}{1+v}\right)^{\alpha_2};x,y
\right)$\newline
$+ M \delta_{n_1}^{\frac{\alpha_1}{2}}(x) \delta_{n_2}^{\frac{\alpha_2}{2}%
}(y).$

Using the Hölder inequality for $p_1= \frac{1}{\alpha_1},~~~q_1 =\frac{1}{%
1-\alpha_1}$, and $p_2= \frac{1}{\alpha_2},~~~q_2 =\frac{1}{1-\alpha_2}$ we
get\newline
\newline

$\bigg{|} L^{(p_{n_1},p_{n_2}),(q_{n_1},q_{n_2}),(\gamma_1,%
\gamma_2)}_{n_1,n_2}(f;x,y)-f(x,y) \bigg{|}$\newline

$\leq M \left( \frac{[n_1]_{p_{n_1},q_{n_1}}}{c_{n_1} +\gamma_1}%
\right)^{\alpha_1} \left( \frac{\gamma_1}{[n_1]_{p_{n_1},q_{n_1}}}%
\right)^{\alpha_1} \left( \frac{[n_2]_{p_{n_2},q_{n_2}}}{c_{n_2} +\gamma_2}%
\right)^{\alpha_2} \left( \frac{\gamma_2}{[n_2]_{p_{n_2},q_{n_2}}}%
\right)^{\alpha_2}$\newline

$+M \bigg{|} 1- \frac{[n_1+1]_{p_{n_1},q_{n_1}}}{c_{n_1}+\gamma_1}\bigg{|}%
^{\alpha_1} \bigg{|} 1- \frac{[n_2+1]_{p_{n_2},q_{n_2}}}{c_{n_2}+\gamma_2}%
\bigg{|}^{\alpha_2} L_{p_{n_1},q_{n_1},p_{n_2},q_{n_2}}^{n_1,n_2}\left(
\frac{u}{1+u} ;x,y \right)^{\alpha_1} \left(
L_{p_{n_1},q_{n_1},p_{n_2},q_{n_2}}^{n_1,n_2}(e_{00};x,y)\right)^{1-%
\alpha_1} $\newline

$\times L_{p_{n_1},q_{n_1},p_{n_2},q_{n_2}}^{n_1,n_2}\left( \frac{v}{1+v}
;x,y \right)^{\alpha_2} \left(
L_{p_{n_1},q_{n_1},p_{n_2},q_{n_2}}^{n_1,n_2}(e_{00};x,y)\right)^{1-%
\alpha_2} + M \delta_{n_1}^{\frac{\alpha_1}{2}}(x) \delta_{n_2}^{\frac{%
\alpha_2}{2}}(y) $\newline

$\leq M \left( \frac{[n_1]_{p_{n_1},q_{n_1}}}{c_{n_1} +\gamma_1}%
\right)^{\alpha_1} \left( \frac{\gamma_1}{[n_1]_{p_{n_1},q_{n_1}}}%
\right)^{\alpha_1} \left( \frac{[n_2]_{p_{n_2},q_{n_2}}}{c_{n_2} +\gamma_2}%
\right)^{\alpha_2} \left( \frac{\gamma_2}{[n_2]_{p_{n_2},q_{n_2}}}%
\right)^{\alpha_2}$\newline

$+M \bigg{|} 1- \frac{[n_1+1]_{p_{n_1},q_{n_1}}}{c_{n_1}+\gamma_1} \bigg{|}%
^{\alpha_1}\bigg{|} 1- \frac{[n_2+1]_{p_{n_2},q_{n_2}}}{c_{n_2}+\gamma_2} %
\bigg{|}^{\alpha_2} \left( \frac{p_{n_1}[n_1]_{p_{n_1},q_{n_1}}}{%
[n_1+1]_{p_{n_1},q_{n_1}}} \frac{x}{1+x}\right)^{\alpha_1} \left( \frac{%
p_{n_2}[n_2]_{p_{n_2},q_{n_2}}}{[n_2+1]_{p_{n_2},q_{n_2}}} \frac{y}{1+y}%
\right)^{\alpha_2} + M \delta_{n_1}^{\frac{\alpha_1}{2}}(x) \delta_{n_2}^{%
\frac{\alpha_2}{2}}(y).$

This completes the proof.
\end{proof}

\end{document}